\documentclass{amsart}
\usepackage{amssymb}
\usepackage{amsfonts}

\newtheorem{theorem}{Theorem}
\theoremstyle{plain}

\newtheorem{corollary}{Corollary}

\newtheorem{lemma}{Lemma}

\newtheorem{remark}{Remark}

\numberwithin{equation}{section}
\input{tcilatex}

\begin{document}
\title[Two Theorems for $\left( \alpha ,m\right) -$convex functions]{Two
Theorems for $\left( \alpha ,m\right) -$convex functions}
\author{M.Emin Ozdemir$^{\blacklozenge }$}
\address{$^{\blacklozenge }$ATATURK UNIVERSITY, K.K. EDUCATION FACULTY,
DEPARTMENT OF MATHEMATICS, 25240 CAMPUS, ERZURUM, TURKEY}
\email{emos@atauni.edu.tr}
\author{Merve Avci Ardic$^{\ast ,\diamondsuit }$}
\thanks{$^{\diamondsuit }$Corresponding Author}
\address{$^{\ast }$ADIYAMAN UNIVERSITY, FACULTY OF SCIENCE AND ART,
DEPARTMENT OF MATHEMATICS, 02040, ADIYAMAN, TURKEY}
\email{mavci@posta.adiyama.edu.tr}

\begin{abstract}
In this paper, we obtain some new inequalities for $\left( \alpha ,m\right) -
$convex functions. The analysis used in the proofs is fairly elementary and
based on the use of Power-mean inequality.
\end{abstract}

\maketitle

\section{INTRODUCTION}

Throughout this paper, let $\left\Vert g\right\Vert _{\infty }=\sup_{t\in
\lbrack a,b]}\left\vert g(t)\right\vert .$

The following inequality is well known in the literature as the
Hermite-Hadamard integral inequality:%
\begin{equation}
f\left( \frac{a+b}{2}\right) \leq \frac{1}{b-a}\int_{a}^{b}f(x)dx\leq \frac{%
f(a)+f(b)}{2}  \label{1.0}
\end{equation}%
where $f:I\subseteq 
\mathbb{R}
\rightarrow 
\mathbb{R}
$ is a convex function on the interval $I$ of real numbers and $a,b\in I$
with $a<b.$

In \cite{M}, Mihe\c{s}an introduced the class of $\left( \alpha ,m\right) -$%
convex functions as the following:

The function $f:[0,b]\rightarrow 
\mathbb{R}
$ is said to be $\left( \alpha ,m\right) -$convex, where $\left( \alpha
,m\right) \in \lbrack 0,1]^{2},$ if for every $x,y\in \lbrack 0,b]$ and $%
t\in \lbrack 0,1],$ we have%
\begin{equation*}
f(tx+m(1-t)y)\leq t^{\alpha }f(x)+m(1-t^{\alpha })f(y).
\end{equation*}

Note that for $\left( \alpha ,m\right) \in \left\{ \left( 0,0\right) ,\left(
\alpha ,0\right) ,\left( 1,0\right) ,\left( 1,m\right) ,\left( 1,1\right)
,\left( \alpha ,1\right) \right\} $ one obtains the following classes of
functions: increasing, $\alpha -$starshaped, starshaped, $m-$convex, convex
and $\alpha -$convex.

For some results concerning $m-$convex and $\left( \alpha ,m\right) -$convex
functions see \cite{MBP}-\cite{MMO}.

In \cite{THYL}, Tseng et al. used the following identities to prove their
results:

\begin{lemma}
\label{lem 1.1} Let $f:I^{\circ }\subseteq 
\mathbb{R}
\rightarrow 
\mathbb{R}
$ be a differentiable mapping on $I^{\circ },$ $a,b\in I^{\circ }$ with $a<b$
and let $g:[a,b]\rightarrow 
\mathbb{R}
$. If $f^{\prime },g\in L[a,b],$ then , for all $x\in \lbrack a,b],$ the
following identity holds:%
\begin{equation}
f(a)\int_{a}^{x}g(s)ds+f(b)\int_{x}^{b}g(s)ds-\int_{a}^{b}f(s)g(s)ds=%
\int_{a}^{b}\left[ \int_{x}^{t}g(s)ds\right] f^{\prime }(t)dt.  \label{1.7}
\end{equation}
\end{lemma}

\begin{lemma}
\label{lem 1.2} Under the assumptions of Lemma \ref{lem 1.2}, we have the
idedntity%
\begin{equation*}
f(x)\int_{a}^{b}g(s)ds-\int_{a}^{b}f(s)g(s)ds=\int_{a}^{b}S_{g}(t)f^{\prime
}(t)dt
\end{equation*}%
where 
\begin{equation*}
S_{g}(t)=\left\{ 
\begin{array}{c}
\int_{a}^{t}g(s)ds,\text{ \ \ \ \ }t\in \lbrack a,x) \\ 
\\ 
-\int_{t}^{b}g(s)ds,\text{ \ \ \ \ }t\in \lbrack x,b].%
\end{array}%
\right.
\end{equation*}%
Define 
\begin{equation*}
S(t)=\left\{ 
\begin{array}{c}
t-a,\text{ \ }t\in \lbrack a,x) \\ 
b-t,\text{ \ \ }t\in \lbrack x,b]\text{\ .\ \ \ }%
\end{array}%
\right.
\end{equation*}%
Then $\left\vert S_{g}(t)\right\vert \leq \left\Vert g\right\Vert _{\infty
}S(t),$ $t\in \lbrack a,b].$
\end{lemma}

\begin{theorem}
\label{teo 1.3}Let $f:I^{\circ }\subseteq 
\mathbb{R}
\rightarrow 
\mathbb{R}
$ be a differentiable mapping on $I^{\circ },$ $a,b\in I^{\circ }$ with $%
a<b,f^{\prime }\in L[a,b],$ $q\geq 1,$ $\left\vert f^{\prime }\right\vert
^{q}$ is convex on $[a,b]$ and $g:[a,b]\rightarrow 
\mathbb{R}
$ be a continuous mapping on $[a,b]$. Then, for all $x\in \lbrack a,b],$ we
have the inequality:%
\begin{eqnarray}
&&\left\vert
f(a)\int_{a}^{x}g(s)ds+f(b)\int_{x}^{b}g(s)ds-\int_{a}^{b}f(s)g(s)ds\right%
\vert  \label{1.3} \\
&\leq &\left\Vert g\right\Vert _{\infty }\left[ \frac{\left( x-a\right)
^{2}+\left( b-x\right) ^{2}}{2}\right] ^{\frac{q-1}{q}}  \notag \\
&&\times \left\{ \frac{\left( x-a\right) ^{2}\left( 3b-x-2a\right) +\left(
b-x\right) ^{3}}{6\left( b-a\right) }\left\vert f^{\prime }(a)\right\vert
^{q}\right.  \notag \\
&&\left. +\frac{\left( x-a\right) ^{3}+\left( b-x\right) ^{2}\left(
2b+x-3a\right) }{6\left( b-a\right) }\left\vert f^{\prime }(b)\right\vert
^{q}\right\} ^{\frac{1}{q}}.  \notag
\end{eqnarray}
\end{theorem}

\begin{corollary}
\label{co 1.1} Let $g:[a,b]\rightarrow 
\mathbb{R}
$ symmetric to $\frac{a+b}{2}$ and $x=\frac{a+b}{2}$ in Theorem \ref{teo 1.3}%
. Then we have the inequality%
\begin{eqnarray}
&&\left\vert \frac{f(a)+f(b)}{2}\int_{a}^{b}g(s)ds-\int_{a}^{b}f(u)g(s)ds%
\right\vert  \label{1.4} \\
&\leq &\left\Vert g\right\Vert _{\infty }\frac{\left( b-a\right) ^{2}}{4}%
\left[ \frac{\left\vert f^{\prime }(a)\right\vert ^{q}+\left\vert f^{\prime
}(b)\right\vert ^{q}}{2}\right] ^{\frac{1}{q}}  \notag
\end{eqnarray}%
which is the "weighted trapezoid" inequality.
\end{corollary}

\begin{theorem}
\label{teo 1.4} Let the assumptions of Theorem \ref{teo 1.3} hold. Then, for
all $x\in \lbrack a,b],$ we have the inequality:%
\begin{eqnarray}
&&  \label{1.5} \\
&&\left\vert f(x)\int_{a}^{b}g(s)ds-\int_{a}^{b}f(s)g(s)ds\right\vert  \notag
\\
&\leq &\left\Vert g\right\Vert _{\infty }\left[ \frac{\left( x-a\right)
^{2}+\left( b-x\right) ^{2}}{2}\right] ^{\frac{q-1}{q}}  \notag \\
&&\times \left\{ \frac{\left( x-a\right) ^{2}\left( 3b-a-2x\right) +2\left(
b-x\right) ^{3}}{6\left( b-a\right) }\left\vert f^{\prime }(a)\right\vert
^{q}\right.  \notag \\
&&\left. +\frac{2\left( x-a\right) ^{3}+\left( b-x\right) ^{2}\left(
b+2x-3a\right) }{6\left( b-a\right) }\left\vert f^{\prime }(b)\right\vert
^{q}\right\} ^{\frac{1}{q}}.  \notag
\end{eqnarray}
\end{theorem}

\begin{corollary}
\label{co 1.2} Let $x=\frac{a+b}{2}$ in Theorem \ref{teo 1.4}. Then we have
the inequality%
\begin{eqnarray}
&&\left\vert f\left( \frac{a+b}{2}\right)
\int_{a}^{b}g(s)ds-\int_{a}^{b}f(s)g(s)ds\right\vert  \label{1.6} \\
&\leq &\left\Vert g\right\Vert _{\infty }\frac{\left( b-a\right) ^{2}}{4}%
\left[ \frac{\left\vert f^{\prime }(a)\right\vert ^{q}+\left\vert f^{\prime
}(b)\right\vert ^{q}}{2}\right] ^{\frac{1}{q}}  \notag
\end{eqnarray}%
which is the "weighted midpoint" inequality provided that $\left\vert
f^{\prime }\right\vert ^{q}$ is convex on $\left[ a,b\right] .$
\end{corollary}

\section{MAIN RESULTS}

Our main results are given in the following theorems.

\begin{theorem}
\label{teo 2.1} Let $f:I\subseteq \lbrack 0,b^{\ast }]\rightarrow 
\mathbb{R}
$ be a differentiable mapping on $I^{\circ }$ such that $f^{\prime }\in
L[a,b],$ $a,b\in I^{\circ }$ with $a<b$, $b^{\ast }>0$ and let $%
g:[a,b]\rightarrow 
\mathbb{R}
$ be a continuous mapping on $[a,b].$ If $\left\vert f^{\prime }\right\vert
^{q}$ is $\left( \alpha ,m\right) -$convex on $[a,b]$ for $\left( \alpha
,m\right) \in (0,1]^{2},q\geq 1,$ then, for all $x\in \lbrack a,b],$ we have
the inequality: 
\begin{eqnarray*}
&&\left\vert
f(a)\int_{a}^{x}g(u)du+f(b)\int_{x}^{b}g(u)du-\int_{a}^{b}f(u)g(u)du\right%
\vert  \\
&\leq &\left\Vert g\right\Vert _{\infty }\left[ \frac{\left( x-a\right)
^{2}+\left( b-x\right) ^{2}}{2}\right] ^{\frac{q-1}{q}} \\
&&\times \left\{ M\left\vert f^{\prime }(a)\right\vert ^{q}+m\left[ \frac{%
\left( x-a\right) ^{2}+\left( b-x\right) ^{2}}{2}-M\right] \left\vert
f^{\prime }\left( \frac{b}{m}\right) \right\vert ^{q}\right\} ^{\frac{1}{q}},
\end{eqnarray*}%
where $M=\frac{\left( b-a\right) ^{\alpha +1}\left[ 2x-b-a+\alpha \left(
x-a\right) \right] +2\left( b-x\right) ^{\alpha +2}}{(\alpha +1)(\alpha
+2)\left( b-a\right) ^{\alpha }}.$
\end{theorem}

\begin{proof}
Using Lemma \ref{lem 1.1}, property of modulus, Power-mean integral
inequality and the $\left( \alpha ,m\right) -$convexity of $\left\vert
f^{\prime }\right\vert ^{q}$, it follows that%
\begin{eqnarray*}
&&\left\vert
f(a)\int_{a}^{x}g(u)du+f(b)\int_{x}^{b}g(u)du-\int_{a}^{b}f(u)g(u)du\right%
\vert  \\
&=&\left\vert \int_{a}^{b}\left[ \int_{x}^{t}g(u)du\right] f^{\prime
}(t)dt\right\vert  \\
&\leq &\int_{a}^{b}\left\vert \int_{x}^{t}g(u)duf^{\prime }(t)\right\vert dt
\\
&\leq &\left[ \int_{a}^{b}\left\vert \int_{x}^{t}g(u)du\right\vert dt\right]
^{\frac{q-1}{q}}\left[ \int_{a}^{b}\left\vert \int_{x}^{t}g(u)du\right\vert
\left\vert f^{\prime }(t)\right\vert ^{q}dt\right] ^{\frac{1}{q}} \\
&\leq &\left\Vert g\right\Vert _{\infty }\left[ \int_{a}^{b}\left\vert
t-x\right\vert dt\right] ^{\frac{q-1}{q}}\left[ \int_{a}^{b}\left\vert
t-x\right\vert \left\vert f^{\prime }(t)\right\vert ^{q}dt\right] ^{\frac{1}{%
q}} \\
&=&\left\Vert g\right\Vert _{\infty }\left[ \frac{\left( x-a\right)
^{2}+\left( b-x\right) ^{2}}{2}\right] ^{\frac{q-1}{q}}\left[
\int_{a}^{b}\left\vert t-x\right\vert \left\vert f^{\prime }\left( \frac{b-t%
}{b-a}a+\frac{t-a}{b-a}b\right) \right\vert ^{q}dt\right] ^{\frac{1}{q}} \\
&\leq &\left\Vert g\right\Vert _{\infty }\left[ \frac{\left( x-a\right)
^{2}+\left( b-x\right) ^{2}}{2}\right] ^{\frac{q-1}{q}} \\
&&\times \left[ \int_{a}^{b}\left\vert t-x\right\vert \left[ \left( \frac{b-t%
}{b-a}\right) ^{\alpha }\left\vert f^{\prime }(a)\right\vert ^{q}+m\left(
1-\left( \frac{b-t}{b-a}\right) ^{\alpha }\right) \left\vert f^{\prime
}\left( \frac{b}{m}\right) \right\vert ^{q}\right] dt\right] ^{\frac{1}{q}}
\end{eqnarray*}%
where 
\begin{equation*}
\int_{a}^{b}\left\vert t-x\right\vert \left( \frac{b-t}{b-a}\right) ^{\alpha
}dt=\frac{\left( b-a\right) ^{\alpha +1}\left[ 2x-b-a+\alpha \left(
x-a\right) \right] +2\left( b-x\right) ^{\alpha +2}}{(\alpha +1)(\alpha
+2)\left( b-a\right) ^{\alpha }}
\end{equation*}%
and%
\begin{eqnarray*}
&&\int_{a}^{b}\left\vert t-x\right\vert \left( 1-\left( \frac{b-t}{b-a}%
\right) ^{\alpha }\right) dt \\
&=&\frac{\left( x-a\right) ^{2}+\left( b-x\right) ^{2}}{2}-\frac{\left(
b-a\right) ^{\alpha +1}\left[ 2x-b-a+\alpha \left( x-a\right) \right]
+2\left( b-x\right) ^{\alpha +2}}{(\alpha +1)(\alpha +2)\left( b-a\right)
^{\alpha }}.
\end{eqnarray*}%
The proof is completed.
\end{proof}

\begin{corollary}
\label{co 2.1} Let $g:[a,b]\rightarrow 
\mathbb{R}
$ symmetric to $\frac{a+b}{2}$ and $x=\frac{a+b}{2}$ in Theorem \ref{teo 2.1}%
. Then we have the inequality%
\begin{eqnarray*}
&&\left\vert \frac{f(a)+f(b)}{2}\int_{a}^{b}g(u)du-\int_{a}^{b}f(u)g(u)du%
\right\vert \\
&\leq &\left\Vert g\right\Vert _{\infty }\left[ \frac{1}{\left( \alpha
+1\right) \left( \alpha +2\right) }\right] ^{\frac{1}{q}}\frac{\left(
b-a\right) ^{2}}{4^{1-\frac{1}{q}}} \\
&&\times \left[ \frac{\alpha 2^{\alpha }+1}{2^{\alpha +1}}\left\vert
f^{\prime }(a)\right\vert ^{q}+m\frac{2^{\alpha }\left( \alpha ^{2}+\alpha
+2\right) -2}{2^{\alpha +2}}\left\vert f^{\prime }\left( \frac{b}{m}\right)
\right\vert ^{q}\right] ^{\frac{1}{q}}.
\end{eqnarray*}
\end{corollary}

\begin{remark}
\label{rem 2.1} In Corollary \ref{co 2.1}, if we choose $\alpha =m=1,$ we
have the inequality in (\ref{1.4}).
\end{remark}

\begin{remark}
\label{rem 2.2} In Theorem \ref{teo 2.1}, if we choose $\alpha =m=1,$ we
have the inequality in (\ref{1.3}).
\end{remark}

\begin{theorem}
\label{teo 2.2} Let the assumptions of Theorem \ref{teo 2.1} hold. Then, for
all $x\in \lbrack a,b],$ we have the inequality:%
\begin{eqnarray*}
&&\left\vert f(x)\int_{a}^{b}g(u)du-\int_{a}^{b}f(u)g(u)du\right\vert  \\
&\leq &\left\Vert g\right\Vert _{\infty }\left[ \frac{\left( x-a\right)
^{2}+\left( b-x\right) ^{2}}{2}\right] ^{\frac{q-1}{q}} \\
&&\times \left\{ A\left\vert f^{\prime }(a)\right\vert ^{q}+m\left[ \frac{%
\left( x-a\right) ^{2}+\left( b-x\right) ^{2}}{2}-A\right] \left\vert
f^{\prime }\left( \frac{b}{m}\right) \right\vert ^{q}\right\} ^{\frac{1}{q}},
\end{eqnarray*}%
where $A=\frac{\left( b-a\right) ^{\alpha +2}+\left( b-x\right) ^{\alpha +1}%
\left[ \left( a-x\right) \left( 2+\alpha \right) +\alpha \left( b-x\right) %
\right] }{(\alpha +1)(\alpha +2)\left( b-a\right) ^{\alpha }}.$
\end{theorem}

\begin{proof}
Using the identities in Lemma \ref{lem 1.2}, Power-mean integral inequality
and the $\left( \alpha ,m\right) -$convexity of $\left\vert f^{\prime
}\right\vert ^{q},$ it follows that%
\begin{eqnarray*}
&&\left\vert f(x)\int_{a}^{b}g(u)du-\int_{a}^{b}f(u)g(u)du\right\vert  \\
&\leq &\int_{a}^{b}\left\vert S_{g}(t)\right\vert \left\vert f^{\prime
}(t)\right\vert dt \\
&\leq &\left\Vert g\right\Vert _{\infty }\int_{a}^{b}S(t)\left\vert
f^{\prime }(t)\right\vert dt \\
&\leq &\left\Vert g\right\Vert _{\infty }\left[ \int_{a}^{b}S(t)dt\right] ^{%
\frac{q-1}{q}} \\
&&\times \left[ \int_{a}^{x}\left( t-a\right) \left\vert f^{\prime }\left( 
\frac{b-t}{b-a}a+\frac{t-a}{b-a}b\right) \right\vert
^{q}dt+\int_{x}^{b}\left( b-t\right) \left\vert f^{\prime }\left( \frac{b-t}{%
b-a}a+\frac{t-a}{b-a}b\right) \right\vert ^{q}dt\right] ^{\frac{1}{q}} \\
&\leq &\left\Vert g\right\Vert _{\infty }\left[ \int_{a}^{b}S(t)dt\right] ^{%
\frac{q-1}{q}} \\
&&\times \left\{ \int_{a}^{x}\left( t-a\right) \left( \frac{b-t}{b-a}\right)
^{\alpha }\left\vert f^{\prime }(a)\right\vert
^{q}dt+\int_{x}^{b}(b-t)\left( \frac{b-t}{b-a}\right) ^{\alpha }\left\vert
f^{\prime }\left( a\right) \right\vert ^{q}dt\right.  \\
&&\left. +\int_{a}^{x}\left( t-a\right) m\left( 1-\left( \frac{b-t}{b-a}%
\right) ^{\alpha }\right) \left\vert f^{\prime }\left( \frac{b}{m}\right)
\right\vert ^{q}dt+\int_{x}^{b}\left( b-t\right) m\left( 1-\left( \frac{b-t}{%
b-a}\right) ^{\alpha }\right) \left\vert f^{\prime }\left( \frac{b}{m}%
\right) \right\vert ^{q}dt\right\} ^{^{\frac{1}{q}}}
\end{eqnarray*}%
where%
\begin{equation*}
\int_{a}^{b}S(t)dt=\frac{\left( x-a\right) ^{2}+\left( b-x\right) ^{2}}{2},
\end{equation*}%
\begin{equation*}
\int_{a}^{x}\left( t-a\right) \left( \frac{b-t}{b-a}\right) ^{\alpha }dt=%
\frac{\left( b-a\right) ^{\alpha +2}+\left( b-x\right) ^{\alpha +1}\left[
2a-b-x+\alpha \left( a-x\right) \right] }{\left( b-a\right) ^{\alpha
}(\alpha +1)(\alpha +2)},
\end{equation*}%
\begin{equation*}
\int_{x}^{b}\frac{(b-t)^{\alpha +1}}{\left( b-a\right) ^{\alpha }}dt=\frac{%
\left( b-x\right) ^{\alpha +2}}{\left( b-a\right) ^{\alpha }\left( \alpha
+2\right) },
\end{equation*}%
\begin{equation*}
\int_{x}^{b}\left( b-t\right) m\left( 1-\left( \frac{b-t}{b-a}\right)
^{\alpha }\right) dt=m\left[ \frac{\left( b-x\right) ^{2}}{2}-\frac{\left(
b-x\right) ^{\alpha +2}}{\left( b-a\right) ^{\alpha }\left( \alpha +2\right) 
}\right] 
\end{equation*}%
and%
\begin{eqnarray*}
&&\int_{a}^{x}\left( t-a\right) \left( 1-\left( \frac{b-t}{b-a}\right)
^{\alpha }\right) dt \\
&=&m\left[ \frac{\left( x-a\right) ^{2}}{2}-\frac{\left( b-a\right) ^{\alpha
+2}+\left( b-x\right) ^{\alpha +1}\left[ 2a-b-x+\alpha \left( a-x\right) %
\right] }{\left( b-a\right) ^{\alpha }(\alpha +1)(\alpha +2)}\right] .
\end{eqnarray*}%
The proof is completed.
\end{proof}

\begin{corollary}
\label{co 2.2} Let $x=\frac{a+b}{2}$ in Theorem \ref{teo 2.2}. Then we have
the following inequality:%
\begin{eqnarray*}
&&\left\vert f\left( \frac{a+b}{2}\right)
\int_{a}^{b}g(u)du-\int_{a}^{b}f(u)g(u)du\right\vert \\
&\leq &\left\Vert g\right\Vert _{\infty }\frac{\left( b-a\right) ^{2}}{4^{1-%
\frac{1}{q}}}\left[ \frac{1}{(\alpha +1)(\alpha +2)}\right] ^{\frac{1}{q}} \\
&&\times \left\{ \left[ \frac{2^{\alpha +1}-1}{2^{\alpha +1}}\right]
\left\vert f^{\prime }(a)\right\vert ^{q}+m\left[ \frac{2^{\alpha }\left[
\alpha ^{2}+3\alpha -2\right] +2}{2^{\alpha +2}}\right] \left\vert f^{\prime
}\left( \frac{b}{m}\right) \right\vert ^{q}\right\} ^{\frac{1}{q}}.
\end{eqnarray*}

\begin{remark}
\label{rem 2.3} In Corollary \ref{co 2.2}, if we choose $\alpha =m=1,$ we
have the inequality in (\ref{1.6}).
\end{remark}

\begin{remark}
\label{rem 2.4} In Theorem \ref{teo 2.2}, if we choose $\alpha =m=1,$ we
have the inequality in (\ref{1.5}).
\end{remark}

\begin{remark}
\label{rem 2.5} In Theorems \ref{teo 2.1}-\ref{teo 2.2} and Corollaries \ref%
{co 2.1}-\ref{co 2.2}, if we choose $\alpha =1$ we obtain inequalities for $%
m-$convex functions.
\end{remark}
\end{corollary}


\begin{thebibliography}{9}
\bibitem{THYL} K.-L. Tseng, S.-R. Hwang, G.-S. Yang and J.-C. Lo, Two
inequalities for differentiable mappings and applications to weighted
trapezoidal formula, weighted midpoint formula and random variable,\textit{\
Mathematical and Computer Modelling}, 53 (2011) 179--188.

\bibitem{M} V.G. Mihe\c{s}an, A generalization of the convexity, Seminar on
Functional Equations, Approx. and Convex., Cluj-Napoca, Romania, 1993.

\bibitem{MBP} M. K. Bakula, M. E. \"{O}zdemir and J. Pe\v{c}ari\'{c},
Hadamard type inequalities for $m-$convex and $\left( \alpha ,m\right) -$%
convex functions, JIPAM, Vol. 9 (2008), Issue 4, Art. 96, 12 pp.

\bibitem{BPP} M. K. Bakula, J. E. Pe\v{c}ari\'{c} and M. Ribi\v{c}i\'{c},
Companion inequalities to Jensen' s inequality for $m-$convex and $\left(
\alpha ,m\right) -$convex functions, JIPAM, Vol. 7 (2006), Issue 5, Art. 194.

\bibitem{SSOR} E. Set, M. Sardari, M. E. \"{O}zdemir and J. Rooin, On
generalizations of the Hadamard inequality for $\left( \alpha ,m\right) -$%
convex functions, Kyungpook Mathematical Journal, Vol. 52 (2012) , No. 3,
Art. 7.

\bibitem{MO} M.Z. Sar\i kaya, E. Set, M.E. \"{O}zdemir, Some new Hadamard's
type inequalities for co-ordinated $m-$convex and $\left( \alpha ,m\right) -$%
convex functions, Hacettepe Journal of Mathematics and Statistics, 40,
219-229, (2011).

\bibitem{MMO} M.E. \"{O}zdemir, M. Avci, E. Set, On some inequalities of
Hermite--Hadamard type via $m-$convexity, Appl. Math. Lett. 23 (9) (2010)
1065--1070.
\end{thebibliography}
\end{document}